\documentclass[reqno, oneside]{amsart}
\usepackage{amsfonts, amsmath, amsthm, amssymb}
\usepackage{mathrsfs}
\usepackage[shortlabels]{enumitem}
\usepackage{geometry}
\usepackage[utf8]{inputenc}
\usepackage{nicefrac}
\usepackage{tikz-cd}

\usepackage{natbib}
\usepackage{adjustbox}
\usepackage{xspace}
\usepackage{etoolbox}
\usepackage{booktabs}
\usepackage[hidelinks]{hyperref}
\usepackage[capitalise,nosort]{cleveref}
\usepackage[colorinlistoftodos]{todonotes}
\usepackage{comment}

\newtheorem{theorem}{Theorem}[section]
\newtheorem{lemma}[theorem]{Lemma}
\newtheorem{proposition}[theorem]{Proposition}

\theoremstyle{definition}
\newtheorem{definition}[theorem]{Definition}
\newtheorem{remark}[theorem]{Remark}

\newtheorem{example}[theorem]{Example}

\crefformat{footnote}{#2\footnotemark[#1]#3}

\newcommand{\ReDeclareMathOperator}[2]{\let#1\relax\DeclareMathOperator{#1}{#2}}

\ReDeclareMathOperator{\int}{int}

\ReDeclareMathOperator{\K}{K}
\ReDeclareMathOperator{\CS}{CSat}
\ReDeclareMathOperator{\KS}{KS}
\ReDeclareMathOperator{\LO}{CLC}
\ReDeclareMathOperator{\CSAT}{CSAT}
\ReDeclareMathOperator{\SFilt}{SFilt}
\ReDeclareMathOperator{\WLC}{WLC}
\ReDeclareMathOperator{\WLO}{WCLC}
\ReDeclareMathOperator{\GC}{AC}
\ReDeclareMathOperator{\P}{\mathcal{P}}

\newcommand{\categoryname}[1]{\ensuremath{\mathbf{#1}}\xspace}
\newcommand{\DeclareCategory}[2]{\newcommand{#1}{\categoryname{#2}}}
\DeclareCategory{\SMT}{SMT}
\DeclareCategory{\wMT}{wMT}
\DeclareCategory{\RMT}{MT_R}
\DeclareCategory{\PMT}{MT_P}
\DeclareCategory{\SRMT}{SRMT}
\DeclareCategory{\RMTz}{RMT_0}
\DeclareCategory{\SRMTz}{SRMT_0}
\DeclareCategory{\KMT}{KMT}
\DeclareCategory{\MT}{MT}
\DeclareCategory{\Frm}{Frm}
\DeclareCategory{\Raney}{RE}
\DeclareCategory{\pRaney}{pRaney}
\DeclareCategory{\BFrm}{BoolFrm}
\DeclareCategory{\CFrm}{CFrm}
\DeclareCategory{\SFrm}{SFrm}
\DeclareCategory{\SHFrm}{SHFrm}
\DeclareCategory{\NFrm}{NormFrm}
\DeclareCategory{\SNFrm}{SNormFrm}
\DeclareCategory{\HFrm}{HFrm}
\DeclareCategory{\RegFrm}{RegFrm}
\DeclareCategory{\SRegFrm}{SRegFrm}
\DeclareCategory{\CRegFrm}{CRegFrm}
\DeclareCategory{\SCRegFrm}{SCRegFrm}
\DeclareCategory{\KRegFrm}{KRegFrm}
\DeclareCategory{\Top}{Top}
\DeclareCategory{\Topz}{Top_0}
\DeclareCategory{\Sob}{Sob}
\DeclareCategory{\IA}{IA}
\DeclareCategory{\Reg}{Reg}
\DeclareCategory{\Set}{Set}
\DeclareCategory{\CABA}{CABA}
\DeclareCategory{\HA}{HA}
\DeclareCategory{\cHA}{cHA}
\DeclareCategory{\MTP}{MT_P}
\DeclareCategory{\TMTR}{T0 MT_R}
\DeclareCategory{\MTz}{MT_0}
\DeclareCategory{\TzMT}{T0MT}
\DeclareCategory{\Topwlc}{Top_{WLC}}
\DeclareCategory{\SzMT}{ST0MT}
\DeclareCategory{\TDMTP}{TD MT_P}

\renewcommand{\diamond}{\lozenge}
\let\ampersand\&
\renewcommand{\&}{\mathbin{\ampersand}}
\newcommand{\upset}{\mathord{\uparrow}}

\newcommand{\ca}[1]{\mathcal{#1}}
\newcommand{\bd}[1]{\mathbf{#1}}

\renewcommand{\O}{\mathsf O}
\newcommand{\C}{\mathsf C}
\renewcommand{\S}{\mathsf{S}}
\newcommand{\LC}{\mathsf{LC}}
\newcommand{\F}{\mathscr F}
\newcommand{\B}{\mathscr B}
\newcommand{\R}{\mathsf{R}}

\newcommand{\ve}{\vee}

\newcommand{\bwe}{\bigwedge}

\newcommand{\se}{\subseteq}

\newcommand{\up}{{\uparrow}}

\newcommand{\Fi}{\mathsf{Filt}}

\newcommand{\fe}{\Fi_{\mathcal{E}}}
\newcommand{\fse}{\Fi_{\mathcal{SE}}}

\setlist[enumerate]{font=\normalfont}

\tikzset{
	symbol/.style={
		draw=none,
		every to/.append style={
			edge node={node [sloped, allow upside down, auto=false]{$#1$}}}
	}
}

\makeatletter
\patchcmd{\@setaddresses}{\indent}{\noindent}{}{}
\patchcmd{\@setaddresses}{\indent}{\noindent}{}{}
\patchcmd{\@setaddresses}{\indent}{\noindent}{}{}
\patchcmd{\@setaddresses}{\indent}{\noindent}{}{}
\makeatother

\title{McKinsey-Tarski algebras and Raney extensions}
\author{G. Bezhanishvili, R. Raviprakash, A. L. Suarez, J. Walters-Wayland}
\thanks{Corresponding author: A. L. Suarez}

\address{New Mexico State University, Las Cruces, NM, USA}

\email{guram@nmsu.edu}

\address{University of KwaZulu-Natal, South Africa}
\email{raviprakashr@ukzn.ac.za}

\address{University of the Western Cape, South Africa}
\email{annalaurasuarez993@gmail.com}

\address{CECAT, Chapman University, Orange, CA, USA}

\email{jwalterswayland@gmail.com}


\begin{document}
	

	\begin{abstract}
		We introduce the notion of Raney morphism between MT-algebras and show that the resulting category is equivalent to the category of Raney extensions. This is done by generalizing the construction of the Funayama envelope of a frame. The resulting notion of the $T_0$-hull of a Raney extension generalizes that of the $T_D$-hull of a frame. 
	\end{abstract}
	
	\maketitle
	
	\tableofcontents
	\section{Introduction}
	
	The standard approach to pointfree topology is through the formalism of frames or locales \cite{Joh1982,PicadoPultr2012}. But recently more expressive pointfree approaches to space have been developed: the formalism of MT-algebras (McKinsey-Tarski algebras) \cite{BezhanishviliRR2023} and that of Raney 
	extensions 
	\cite{Sua24,suarez25}. As the names suggest, the MT-approach goes back to the work of McKinsey and Tarski \cite{MT1944} and the Raney approach to that of Raney \cite{Ran52}. MT-algebras are complete boolean algebras $B$ equipped with an interior operator $\square$, and can also be thought of as pairs $(B,L)$ such that $B$ is a complete boolean algebra and $L$ is a subframe of $B$ (see \cref{Frames and MT-algebras}). On the other hand, Raney extensions are pairs $(C,L)$, where $C$ is a coframe and $L$ is a subframe of $C$ that meet-generates $C$ and joins in $L$ distribute over binary meets in $C$ (see \cref{sec: MT and RE}). 
	
	There is a close connection between MT-algebras and frames. Indeed, for each MT-algebra $M$, its open elements form a frame $L$ and, up to isomorphism, each frame arises this way. This can be seen by taking the Funayama envelope $\F L$ of $L$ (see \cref{Frames and MT-algebras}). The MT-algebras of the form $\F L$ were characterized in \cite{BezhanishviliRR2023} as those MT-algebras that satisfy the $T_D$-separation axiom. But care is needed with morphisms since not each frame morphism lifts to an MT-morphism between their Funayama envelopes. This was remedied in \cite{bezhanishvili2025funayamaenvelopetdhullframe} where the notion of proximity morphism between MT-algebras was introduced and it was shown that the above one-to-one correspondence lifts to a categorical equivalence. Thus, frames can be thought of as the MT-algebras satisfying the $T_D$-separation, and each frame $L$ has its $T_D$-hull $\F L$.
	
	There is also a close connection between Raney extensions and frames. Indeed, the assignment $(C,L) \mapsto L$ defines a functor from the category $\Raney$ of Raney extensions to the category $\Frm$ of frames, and this functor has a left adjoint \cite{suarez25}. Thus, $\Frm$ can be thought of as a coreflective subcategory of $\Raney$.
	
	It is only natural to compare the two formalisms of MT-algebras and Raney extensions. This was done recently in \cite{GSRA25}, where it was shown that each MT-algebra $M$ gives rise to a Raney extension $\R M := (\S M, \O M)$, where $\S M$ is the coframe of saturated elements and $\O M$ the frame of open elements of $M$. Moreover, up to isomorphism, every Raney extension $R = (C,L)$ arises this way.\footnote{We follow the definition of Raney extensions given in  \cite{GSRA25}, which is a strengthening of that given in \cite{suarez25}.} 
	The latter can be shown by generalizing the Funayama envelope construction to Raney extensions. The MT-algebras of the form $\F R$ were characterized in \cite{GSRA25} as those MT-algebras that satisfy the $T_0$-separation axiom. Thus, the one-to-one correspondence between frames and MT-algebras satisfying the $T_D$-separation extends to a one-to-one correspondence between Raney extensions and MT-algebras satisfying the $T_0$-separation. 
	
	Our aim is to lift this one-to-one correspondence to a categorical equivalence. But, as with frames and MT-algebras, care is needed with morphisms. Indeed, not every Raney morphism lifts to an MT-morphism between their Funayama envelopes, although finding such an example is more involved than in the case of frames (see \cref{sec: equivalence}). We introduce the notion of Raney morphism between MT-algebras, which generalizes that of proximity morphism, and show that the category $\RMT$ of MT-algebras and Raney morphisms is equivalent to $\Raney$. The equivalence is established through the functors $\R : \RMT \to \Raney$ and $\F : \Raney \to \RMT$. In addition, we show that the full subcategory $\TMTR$ of $\RMT$ consisting of $T_0$-algebras is also equivalent to $\Raney$, and hence the reflector $\F\R : \RMT \to \TMTR$ is an equivalence. This counterintuitive phenomenon is explained by the fact that isomorphisms in $\RMT$ are not order-isomorphisms, but this anomaly disappears in $\TMTR$. We thus think of Raney extensions as the MT-algebras satisfying the $T_0$-separation axiom,  generalizing a similar correspondence between frames and the MT-algebras satisfying the $T_D$-separation axiom. In particular, each Raney extension $R$ has the $T_0$-hull $\F R$ generalizing the $T_D$-hull of each frame. 
	
	
	
	\section{Frames and MT-algebras} \label{Frames and MT-algebras}
	
	
	
	
	We recall that a complete lattice $L$ is a {\em frame} if it satisfies the join-infinite distributive law
	\[
	a \wedge \bigvee S = \bigvee \{ a \wedge s \mid s \in S \},
	\]
	and a {\em coframe} if it satisfies the meet-infinite distributive law
	\[
	a \vee \bigwedge S = \bigwedge \{ a \vee s \mid s \in S \}
	\]
	for all $a \in L$ and $S \subseteq L$.
	A {\em frame morphism} is a map between frames preserving arbitrary joins and finite meets; coframe morphisms are defined dually. We let $\Frm$ be the category of frames and frame morphisms.
	
	Standard examples of frames are the lattices $\mathcal O X$ of open sets of topological spaces. Indeed, the predominant approach to pointfree topology is through the category $\Frm$ (and its dual category $\bf Loc$ of locales); see \cite{Joh1982,PicadoPultr2012}. 

		While frames generalize the lattices of open sets of topological spaces, an earlier approach of McKinsey and Tarski \cite{MT1944} (see also \cite{Noeb54}) generalizes closure/interior operators on powerset algebras to arbitrary boolean algebras. This led to the theory of McKinsey-Tarski algebras, which provides an alternative (and more expressive) pointfree approach to topology; see \cite{BezhanishviliRR2023}.  
		
		We recall that an {\em interior operator} on a bounded lattice $L$ is a unary function $\square : L\to L$ satisfying Kuratowski's axioms 
		for all $a,b\in L$:
		\[
		\square 1=1, \quad \square(a \wedge b) = \square a \wedge \square b, \quad \square a \leq a, \quad \mbox{and} \quad \square a \leq \square \square a. 
		\]
		
		\begin{definition}
			$ $
			\begin{enumerate}
				\item A {\em McKinsey-Tarski algebra} or simply an {\em MT-algebra} is a pair $M=(B,\square )$ where $B$ is a
				complete boolean algebra and $\square$ is an interior operator
				on $B$.
				\item An {\em MT-morphism}  between MT-algebras $M$ and $N$ is a complete boolean morphism
				$f : M \to N$ such that $f(\square a) \leq \square f(a)$ for each $a \in M$.
				\item  Let $\MT$ be the category of MT-algebras and MT-morphisms.
			\end{enumerate}
		\end{definition}
		
		MT-algebras can alternatively be defined as pairs $(B,L)$ where $B$ is a complete boolean algebra and $L$ is a subframe of $B$. Indeed, given an MT-algebra $(B,\square)$, the set 
		\[
		L := \{ a \in B \mid a = \square a \}
		\]
		of fixpoints of $\square$ is a subframe of $B$. Moreover, every subframe $L$ of $B$ is the subframe of fixpoints of the right adjoint $\Box : B \to L$ of the embedding $e : L \to B$. Furthermore, a complete boolean morphism $f : M \to M'$ is an MT-morphism iff its restriction $f: L \to L'$ is well defined (in which case it is a frame morphism between the fixpoints). We thus arrive at the following:
		
		\begin{theorem}
			$\MT$ is isomorphic to the category whose objects are pairs $(B,L)$ where $B$ is a complete boolean algebra and $L$ is a subframe of $B$ and whose morphisms are complete boolean morphisms $f: B \to B'$ such that the restriction $f : L \to L'$ is well defined.
		\end{theorem}

			There is a close connection between MT-algebras and frames. For each MT-algebra $M$, let $\O M$ be the fixpoints of $\square$, which we call {\em open elements}. As we pointed out above, $\O M$ is a subframe of $M$, hence $\O M$ is a frame. Moreover, if $f : M \to M'$ is an MT-morphism, then its restriction $f|_{\O M}:\O M\to\O M'$ is a frame morphism. This defines a functor $\O:\MT\to\Frm$. By  \cite[Thm.~4.2]{BezhanishviliRR2023}, $\O$ is essentially surjective. For each frame $L$, the MT-algebra $M$ such that $L \cong \O M$ can be constructed by taking the {\em Funayama envelope} $\F L$ of $L$ \cite{funayama59}. One construction of $\F L$ is to take the MacNeille completion of the boolean envelope of $L$ \cite[Sec.~II.4]{Gra2011}, another is to take the booleanization of the frame of nuclei of $L$ \cite[Sec.~II.2]{Joh1982}, and the two are isomorphic by \cite{BGJ2013}. We will mainly use the former construction. 
			
			We next recall the characterization of MT-algebras which are isomorphic to $\F L$ for some frame $L$.
			For an MT-algebra $M$, let $\diamond := \neg\square\neg$ be the corresponding closure operator. We call $a\in M$ {\em closed} if it is a fixpoint of $\diamond$, and {\em locally closed} if $a = u \wedge c$, where $u$ is open and $c$ is closed. We let $\C M$ denote the closed elements and $\LC M$ the locally closed elements of $M$.
			
			\begin{definition}
				An MT-algebra is 
				a {\em $T_D$-algebra} if each element is a join of locally closed elements.
			\end{definition}
			
			The following result provides the desired characterization:
			\begin{proposition} {\em \cite[Thm.~6.5]{BezhanishviliRR2023}} \label{lem: fom is td}  
				An MT-algebra $M$ is a $T_D$-algebra iff $M \cong \F \O M$.
			\end{proposition}
			
			Nevertheless, taking the Funayama envelope does not lift to a functor from $\Frm$ to $\MT$ since the lift may not be a complete boolean morphism. This was remedied in \cite{bezhanishvili2025funayamaenvelopetdhullframe}, where the notion of proximity morphism between MT-algebras was introduced and the resulting category was shown to be equivalent to $\Frm$. We recall the details below. 
			
			\begin{definition} \label{def: proximitymor}
				A map $f:M \to M'$ between MT-algebras is a \emph{proximity morphism} provided the following conditions are satisfied:
				\begin{enumerate}[label=\upshape(P\arabic*)]
					\item $f| _{\O M} : \O M\to\O N$ is a frame morphism.\label{def: proximitymor 1}
					\item $f(a \wedge b)=f(a) \wedge f(b)$ for each $a,b \in M$.\label{def: proximitymor 2}
					\item $f(\bigvee S)=\bigvee\{f(s) \mid s \in S\}$ for each finite $S \subseteq \LC M$.\label{def: proximitymor 3}
					\item $f(a)=\bigvee\{ f(x) \mid x \in \LC M, \, x \leq a\}$ for each $a \in M$.\label{def: proximitymor 4}
				\end{enumerate}
			\end{definition}
			
			Let $\PMT$ be the category of MT-algebras and proximity morphisms between them. The composition of $f :M_1 \to M_2$ and $g: M_2\to M_3$ in $\PMT$ is defined by 
			\[
			(g \star f)(a) = \bigvee \{ gf(x) \mid x \in \LC M_1, \ x \le a \} 
			\]
			and the identity $id^P_M :M \to M$ by
			\[
			id^P_M(a) = \bigvee \{ x \in \LC M \mid x \le a \}.
			\]
			We also let $\TDMTP$ be the full subcategory of $\MTP$ consisting of $T_D$-algebras. We then have:
			
			\begin{theorem} {\em \cite[Sec.~4]{bezhanishvili2025funayamaenvelopetdhullframe} } \label{thm: TD hull of Frm}
				\begin{enumerate}[ref=\thetheorem(\arabic*)]
					\item 
					$\O:\PMT \to \Frm$ and $\F : \Frm \to \PMT$ are functors, yielding an equivalence of $\PMT$ and $\Frm$. \label[theorem]{MTpequivalenttoframes}
					\item This equivalence restricts to an equivalence between $\TDMTP$ and $\Frm$. Consequently, the reflector $\F\O : \MTP \to \TDMTP$ is an equivalence. \label[theorem]{MTpequivalenttoframes 2}
				\end{enumerate}
			\end{theorem} 
			
			
			
			
			
			The equivalence of $\MTP$ and $\TDMTP$ is explained by the fact that isomorphisms in $\MTP$ are not order-isomorphisms. Indeed, each MT-algebra $M$ is $\MTP$-isomorphic to its $T_D$-reflection $\F\O M$. The situation improves in $\TDMTP$, where isomorphisms are indeed order-isomorphisms (see \cite[Prop.~4.22]{bezhanishvili2025funayamaenvelopetdhullframe}). 
			
			By the above,
			we can identify frames 
			with $T_D$-algebras. In particular, for each frame $L$, we think of $\F L$ as the $T_D$-hull of $L$. 
			
			\section{MT-algebras and Raney extensions} \label{sec: MT and RE}
			
			Another alternative pointfree approach to topology that is more expressive than that of frames is the formalism of Raney extensions \cite{Sua24,suarez25}.
			For a complete lattice $C$,  we say that $L \subseteq C$ is a {\em subframe} of $C$ if $L$ is a frame in the order inherited from $C$ and the embedding $e: L \to C$ preserves arbitrary joins and finite meets.\footnote{Observe that $C$ itself may not be a frame. 
			} 
			
			\begin{definition}{\cite[Sec.~2]{Sua24}}\label{Raneydef}
				\begin{enumerate}
					\item 
					A {\em Raney
						extension} is a pair $R=(C, L)$ such that 
					\begin{enumerate}
						\item $C$ is a coframe. 
						\item $L$ is a subframe of $C$ that meet-generates $C$.
						\item $a \wedge \bigvee S = \bigvee \{ a \wedge s \mid s \in S \}$ for each $a\in C$ and $S\subseteq L$. 
					\end{enumerate}
					\item A {\em morphism} between Raney extensions $R=(C,L)$ and $R'=(C',L')$ is a coframe morphism $f:C \to C'$ such that the restriction $f|_L: L \to L'$ is a well-defined frame morphism.
					\item Let $\Raney$ be the category of Raney extensions and morphisms between them.
				\end{enumerate}
			\end{definition}
			
			\begin{remark}\label{Raneywithinterioroperator}
				Raney extensions can equivalently be defined as pairs $(C,\square)$ where $C$ is a coframe and $\square$ is an interior operator on $C$ such that the fixpoints $L := \{ a \in C \mid a = \square a\}$ satisfy (1b) and (1c). Thus, Raney extensions provide a generalization of Raney algebras of \cite{BH2020}.
				Raney extensions should not be confused with those 
				in \cite{BH2023}, where a different formalism of Raney extensions is introduced to 
				characterize stable compactifications of $T_0$-spaces.
			\end{remark}
			
			
			
			
			
			The close connection between MT-algebras and frames extends to Raney extensions. In a nutshell, if frames can be thought of as MT-algebras satisfying the $T_D$-separation, Raney extensions can be thought of as those MT-algebras that satisfy the $T_0$-separation. To define the latter, for an MT-algebra $M$, we recall that $a \in M$ is {\em saturated} if it is a meet of open elements. Let $\S M$ be the saturated elements of $M$. 
			
			\begin{definition}
				An MT-algebra $M$ is a {\em $T_0$-algebra} if each element is a join of elements of the form $s \wedge c$, where $s\in\S M$ and $c\in \C M$.
			\end{definition}
			
			
			We let $\B \S M$ denote the boolean subalgebra of $M$ generated by $\S M$. Then each element of $\B \S M$ can be written as $
			a = \bigvee_{i=1}^n (s_i \wedge \neg t_i),
			$
			where 
			$s_i, t_i \in \S M $ 
			(see, e.g., \cite[p.~74]{RS1963}). But each $t_i$ is a meet of open elements, so $\neg t_i$ is a join of closed elements, yielding that $\{ s \wedge c \mid s \in\S M, \, c \in \C M\}$ join-generates $\B \S M$. Thus, we obtain:
			
			\begin{proposition}\label{prop: char of T0}
				An MT-algebra $M$ is a $T_0$-algebra iff $\B\S M$ join-generates $M$. 
			\end{proposition}

			Each MT-algebra $M$ gives rise to the Raney extension $\R M := (\S M, \O M)$ (see \cite[Prop.~5.2]{GSRA25}). Conversely, for each Raney extension $R=(C,L)$,
			we can generalize the Funayama envelope construction to produce an MT-algebra. Indeed, let $\F C$ be the MacNeille completion of the boolean envelope of $C$. Since all joins in $L$ distribute over binary meets in $C$, the embedding $L\to\F C$ has a right adjoint, which defines an interior operator $\square$ on $\F C$. Thus, $(\F C,\square)$ is an MT-algebra and $\O(\F C,\square) \cong L$. We call the MT-algebra $(\F C,\square)$ the {\em Funayama envelope} of $R$ and denote it by $\F R$.
			The following result characterizes Funayama envelopes of Raney extensions as $T_0$-algebras:
			
			\begin{proposition} {\em \cite[Thm.~5.4]{GSRA25}} \label{lem: Raney is t0}  
				An MT-algebra $M$ is a $T_0$-algebra iff $M \cong \F \R M$.
			\end{proposition}

			In \cref{sec: equivalence} we will show that this correspondence between MT-algebras and Raney extensions lifts to a categorical equivalence. This requires a generalization of proximity morphisms, which is the subject of next section.

			\section{Raney morphisms between MT-algebras} \label{sec: Raney morphisms}

			As we saw in the previous section, with each MT-algebra $M$ we can associate the Raney extension $\R M = (\S M, \O M)$. If $f : M \to M'$ is an MT-morphism, then it is straightforward to see that the restriction $f|_{\R M} : \R M \to \R M'$ is a well-defined morphism between Raney extensions, 
			thus yielding a functor $\MT \to \Raney$. However, 
			this functor is not full (see \cref{Cantor example}). 
			To remedy this, we introduce a different notion of morphism, which generalizes that of a proximity morphism, between MT-algebras.
			To justify the definition, 
			%
			%
			%
			we recall from \cite{bezhanishvili2025funayamaenvelopetdhullframe} that each boolean subalgebra $B$ of a boolean algebra $A$ gives rise to a proximity-like relation on $A$ given by
			\[
			a \prec_B c \Longleftrightarrow \exists b \in B : a \le b \le c.
			\]
			It is straightforward to verify that this relation satisfies the following conditions:
			\begin{enumerate}
				\item[(S1)] $1\prec_B 1$;
				\item[(S2)] $a\prec_B c$ implies $a\leq c$;
				\item[(S3)] $a\leq a'\prec_B c'\leq c$ implies $a\prec_B c$;
				\item[(S4)] $a\prec_B c,d$ implies $a\prec_B c\wedge d$;
				\item[(S5)] $a\prec_B c$ implies $\neg c\prec_B \neg a$;
				\item[(S6)] $a\prec_B c$ implies that there is $b\in B$ with $a\prec_B b\prec_B c$.
			\end{enumerate}
			
			\begin{remark}\label{rem: de Vries proximity} 
				As was pointed out in \cite[Sec.~3]{bezhanishvili2025funayamaenvelopetdhullframe}, the above axioms are the standard proximity axioms on a boolean algebra, with (S6) being a strengthening of the usual in-betweenness axiom. Moreover, $\prec_B$ is a de~Vries proximity (see \cite{deV62,Bezh2010}) iff 
				$a = \bigvee \{c \in A \mid c \prec_B a\}$, which is equivalent to $B$ join-generating $A$.
			\end{remark}
			
			We will mainly be interested in $\prec_{\B \S M}$, where we 
			recall 
			that $\B \S M$ is the boolean subalgebra of $M$ generated by $\S M$. To simplify notation, we write $\prec$ for $\prec_{\B \S M}$. 
			
			\begin{lemma}\label{lem: proximity and T0}
				Let $M$ be an MT-algebra. The $\prec$ relation is a de Vries proximity on $M$ iff $M$ is a $T_0$-algebra.
			\end{lemma} 
			\begin{proof}
				By \cref{rem: de Vries proximity}, $\prec$ is a de~Vries proximity on $M$ iff $\B \S M$ join-generates $M$. The latter is equivalent to $M$ being a $T_0$-algebra by \cref{prop: char of T0}.
				%
				%
				%
			\end{proof}
			
			
			We are ready to introduce the notion of a Raney morphism between MT-algebras, which is the central concept of this article.
			
			\begin{definition} \label{raneymor}
				A {\em Raney morphism} between MT-algebras is a function $f:M \to M'$ satisfying the following conditions:
				\begin{enumerate}
					[label=\upshape(R\arabic*)]
					\item $ f| _{\S M} : \S M\to\S M'$ is a coframe morphism. \label{def: raneymor 1}
					\item $ f| _{\O M} : \O M\to\O M'$ is a frame morphism.\label{def: raneymor 2}
					\item $f(a \wedge b)=f(a) \wedge f(b)$ for each $a,b \in M$.\label{def: raneymor 3}
					\item $f(x \vee y)= f(x) \vee f(y)$ for each $x, y \in \B \S M$. \label{def: raneymor 4}
					\item $f(a)=\bigvee\{ f(x) \mid x \in \B \S  M, \, x \leq a\}$ for each $a \in M$. \label{def: raneymor 5}
				\end{enumerate}
			\end{definition}
			
			\begin{remark}
				Comparing the above definition to \cref{def: proximitymor}, observe that while each Raney morphism satisfies \ref{def: proximitymor 1}--\ref{def: proximitymor 3}, in general it need not satisfy \ref{def: proximitymor 4}. 
			\end{remark}
			
			\begin{lemma}\label{l:raneymorproperties}
				Let $f:M \to M'$ be a Raney morphism between MT-algebras.
				\begin{enumerate}[ref=\thelemma(\arabic*)]
					\item $f(\neg x)=\neg f(x)$ for each $x \in \S M$.\label[lemma]{l:raneymorproperties 1}
					\item $f|_{\B \S M}:\B \S M \to \B \S M'$ is a boolean morphism. \label[lemma]{l:raneymorproperties 2} 
					\item If $x \in \LC M$ then $f(x) \in \LC M'$. \label[lemma]{l:raneymorproperties3}
				\end{enumerate}
			\end{lemma}
			\begin{proof}
				Let $f:M\to M'$ be a Raney morphism between MT-algebras.
				\begin{enumerate}
					\item  Let $x\in\S M$. Then $x, \neg x \in \B \S M$. Therefore, by \ref{def: raneymor 4} and \ref{def: raneymor 2},
					\[
					f(x) \vee f(\neg x)=f(x \vee \neg x)=f(1)=1.
					\]
					Moreover, by \ref{def: raneymor 3} and \ref{def: raneymor 2}, 
					\[
					f(x) \wedge f(\neg x) = f(x \wedge \neg x) = f(0) = 0.
					\]
					Thus, $f(\neg x)=\neg f(x)$. 
					\item Let $x \in \B \S M$. As we pointed out in the previous section, 
					\[
					x=\bigvee_{i=1}^n\{a_i \wedge \lnot b_i \mid a_i, b_i \in \S M\}.
					\]
					Therefore, 
					by \ref{def: raneymor 4}, \ref{def: raneymor 3}, 
					and \hyperref[l:raneymorproperties 1]{(1)}, we get   \[
					f(x)=\bigvee_{i=1}^n\{f(a_i) \wedge \neg f(b_i) \mid f(a_i), f(b_i) \in \S M'\}.
					\]
					Thus, $f(x) \in \B \S M'$, and so $f|_{\B \S M}$ is well-defined. Moreover, by \ref{def: raneymor 3} and \ref{def: raneymor 4}, it is a lattice morphism, and by \ref{def: raneymor 1} or \ref{def: raneymor 2}, it is bounded. Hence, $f|_{\B \S M}$ is a boolean morphism.
					\item From $x \in \LC M$ it follows that $x = u \wedge \neg v$ with $u,v \in \O M$. But then $f(u), f(v) \in \O M'$ by \ref{def: raneymor 2}. Since $u, v \in \B \S M$, (2) implies that $f(x) = f(u \wedge \neg v) = f(u) \wedge \neg f(v)$. Thus, $f(x) \in \LC M'$. \qedhere
				\end{enumerate}

				%
			\end{proof}
			
			\begin{lemma} \label{lem: eqv morphisms}
				Let $f : M \to M'$ be a map between MT-algebras satisfying \ref{def: raneymor 1}, \ref{def: raneymor 2}, \ref{def: raneymor 3},  and \ref{def: raneymor 5}. The following are equivalent:
				\begin{enumerate}
					[ref=\thelemma(\arabic*)]
					\item $f$ satisfies \ref{def: raneymor 4}; that is, $f$ is a Raney morphism.
					\item $a_1 \prec b_1$ and $a_2 \prec b_2$ imply $f(a_1 \vee a_2) \prec f(b_1) \vee f(b_2)$ for each $a_i,b_i \in M$.
					\item $a \prec b$ implies $\neg f(\neg a) \prec f(b)$ for each $a,b \in M$. \label[lemma]{lem: eqv morphisms 3}
				\end{enumerate}
			\end{lemma}
			
			\begin{proof}
				It is sufficient to prove that (1)$\Leftrightarrow$(2) since (2)$\Leftrightarrow$(3) follows from \cite[Lem.~2.2]{Bezh2012} and \cite[Prop.~7.4]{BezhJohn2014}.
				
				\begin{enumerate}[labelwidth=2em, align=left]
					\item[(1)$\Rightarrow$(2):]  
					Let $a_1 \prec b_1$ and $a_2 \prec b_2$. Then there are $s_1,s_2 \in \B \S M$ such that $a_1 \le s_1 \le b_1$ and $a_2 \le s_2 \le b_2$. Therefore, $a_1 \vee a_2 \le s_1 \vee s_2 \le b_1 \vee b_2$. By \ref{def: raneymor 3}, $f$ is order preserving. Thus, by (1), 
					\[
					f(a_1\vee a_2) \le f(s_1 \vee s_2) = f(s_1) \vee f(s_2) \le f(b_1) \vee f(b_2).
					\]
					Consequently, $f(a_1 \vee a_2) \prec f(b_1) \vee f(b_2)$ since $f(s_1) \vee  f(s_2)\in \B \S M'$ by \cref{l:raneymorproperties 2}.
					
					\item[(2)$\Rightarrow$(1):] Let $x,y\in \B \S M$. Since $f$ is order preserving, $f(x)\vee f(y)\le f(x\vee y)$. For the reverse inequality,  $x\prec x$ and $y\prec y$. Therefore, by (2), $f(x\vee y) \prec f(x) \vee f(y)$, and hence $f(x\vee y)\le f(x)\vee f(y)$. 
					\qedhere 
				\end{enumerate}
				
			\end{proof}
			\begin{definition}
				For Raney morphisms $f:M_1\to M_2$ and $g:M_2\to M_3$, define $g \star f : M_1 \to M_3$ by  
				$
				(g \star f)(a)=\bigvee \{g f(x) \mid x \in \B\S M_1, \, x \leq a \}.
				$
			\end{definition}
			
			It is immediate from the above definition that if $x \in \B \S M_1$ then $(g \star f)(x)=(g \circ f)(x)$.
			
			\begin{lemma} \label{lem: composition}
				Let $f:M_1\to M_2$, $g:M_2\to M_3$, and $h:M_3\to M_4$ be Raney morphisms. For each $a\in M_1$, we have 
				\[
				((h \star g) \star f)(a) = \bigvee \{h g f(x) \mid x \in \B \S M_1, \, x \leq a \} = (h \star (g \star f))(a).
				\]
			\end{lemma}
			\begin{proof}
				Let $a \in M_1$. Then
				\begin{align*}
					((h \star g) \star f)(a) 
					&= \bigvee \{(h\star g)(f(x)) \mid x \in \B \S M_1, \, x \leq a \}\\
					&= \bigvee \{(h \circ g)(f(x)) \mid x \in \B \S M_1, \, x \leq a \} & &\text{since $f(x) \in \B \S M_2$}\\
					&= \bigvee \{h((g \circ f)(x)) \mid x \in \B \S M_1, \, x \leq a \}\\
					&= \bigvee \{h((g \star f)(x)) \mid x \in \B \S M_1, \, x \leq a \}&&\text{since $x\in \B \S M_1$}\\
					&= (h \star (g \star f))(a). && \qedhere
				\end{align*}
			\end{proof}
			
			\begin{definition}
				For an MT-algebra $M$, define $id_M:M\to M$ by
				\[
				id_M(a) = \bigvee \{x \in \B \S M \mid x \leq a\} \ 
				\mbox{ for each }a\in M. \]
			\end{definition}
			
			The next lemma will be used multiple times to prove that certain morphisms preserve finite meets. 
			
			\begin{lemma}\label{l: preservation of all finite meets}
				Let $L$ be a lattice, $L'$ a frame, and $f : L \to L'$ an order preserving map. If  $S\subseteq L$ is closed under binary meets, $f$ preserves all binary meets from $S$, and 
				\[
				f(a)=\bigvee \{f(s)\mid s\in S, \, s\leq a\}
				\]
				for all $a\in L$, then $f$ preserves all binary meets from $L$.
			\end{lemma}
			\begin{proof}
				For $a,b\in L$ we have
				\begin{align*}
					f(a) \wedge f(b) &= \bigvee \{f(x) \mid x \in S, \, x \leq a \} \wedge \bigvee \{f(y) \mid y \in S, \, y \leq b \}\\
					&= \bigvee \{f(x) \wedge f(y) \mid x , y \in S, \, x \leq a, \, y \leq b \} && \text{$L'$ is a frame}\\ 
					&= \bigvee \{f(x \wedge y) \mid x,y \in S, \, x \leq a, \, y \leq b \} && \text{$f$ preserves binary meets from $S$}\\
					&= \bigvee \{f(z) \mid z \in S, \, z \leq a\wedge b \} && \text{$S$ is closed under binary meets}\\
					&= f(a \wedge b). && \qedhere 
				\end{align*}   
			\end{proof}
			
			\begin{lemma}\ \label{lem: identity}
				\begin{enumerate}[ref=\thelemma(\arabic*)]
					\item $id_M$  is a Raney morphism for each MT-algebra $M$. \label{1M}
					\item For each Raney morphism $f:M\to M'$ between MT-algebras,
					\[
					id_{M'} \star f = f = f \star id_M.
					\]
				\end{enumerate}
			\end{lemma}
			\begin{proof}
				(1) We have that $id_M(x) = x$ for each $x \in \B \S M$. Therefore, \ref{def: raneymor 1}, \ref{def: raneymor 2}, and \ref{def: raneymor 5} hold. 
				Since $id_M$ is identity on $\B \S M$, which is closed under binary meets, \cref{l: preservation of all finite meets} applies, yielding that \ref{def: raneymor 3} holds. We show that \cref{lem: eqv morphisms 3} holds. Let $a\prec b$, in particular let $x\in \B\S M$ be such that $a\leq x\leq b$. Since $id_M$ is monotone and $\neg$ is antitone, $\neg id_M(\neg a)\leq \neg id_M(\neg x)$. Because $x\in \B \S M$, $x\leq id_M(b)$. Since also $\neg x\in \B \S M$, $id_M(\neg x)=\neg x$, and so $\neg id_M(\neg x)=x$. This means that $\neg id_M(\neg a)\leq x\leq id_M(b)$, that is, $\neg id_M(\neg a)\prec id_M(b)$. Thus, $id_M$ is a Raney morphism by \cref{lem: eqv morphisms}.
				
				(2) Let $a\in M$. Then
				\begin{align*}
					(id_{M'} \star f)(a) &= \bigvee \{id_{M'} f(x)\mid x\in \B \S M, \, x\leq a\}\\
					&= \bigvee \{ f(x) \mid x\in  \B \S M, \, x\leq a\} & &\text{since $f(x) \in  \B \S M'$}\\
					&= f(a) \\
					&= \bigvee \{f \, id_M(x) \mid x\in \B \S M, \, x\leq a\} & &\text{since $x \in  \B \S M$}\\
					&=(f \star id_M)(a). && \qedhere
				\end{align*} 
				
			\end{proof}
			
			\begin{theorem} \label{composition is proximity mt-morphism}
				The MT-algebras and Raney morphisms form a category, $\RMT$, where composition is given by $\star$ and identity morphisms are $id_M$.
			\end{theorem}
			
			\begin{proof}
				In view of \cref{lem: composition,lem: identity}, it suffices to check that if $f:M_1\to M_2$ and ${g:M_2\to M_3}$ are Raney morphisms, then so is $g \star f : M_1 \to M_3$. For this we verify that $g \star f$ satisfies \ref{def: raneymor 1}--\ref{def: raneymor 5}.
				
				
				\begin{enumerate}[label=\upshape(R\arabic*)]
					\item  For $s \in \S M_1$, we have $(g \star f)(s)= (g \circ f)(s)$. Thus, $(g \star f) | _{\S M_1}$ is a coframe morphism.
					\item For $u \in \O M_1$, we have $(g \star f)(u)= (g \circ f)(u)$. Thus, $(g \star f) | _{\O M_1}$ is a frame morphism.
					\item Since  $(g\star f)(x)=gf(x)$  for each $x\in \B \S M_1$,  
					for all $a\in M_1$ we have \[(g\star f)(a)=\bigvee \{(g\star f)(x)\mid x\in \B \S M_1, \, x\leq a\}.\] Thus, \cref{l: preservation of all finite meets} applies, by which $g\star f$ preserves binary meets.
					\item Let $x,y \in \B \S M_1$. By \ref{def: raneymor 3}, $g \star f$ is order preserving. Thus, 
					\[
					(g \star f)(x) \vee (g \star f)(y) \le (g \star f)(x \vee y).
					\]
					For the reverse inequality, since $(g \star f)(a)\le (g\circ f)(a)$ for each $a \in M_1$ and $f,g$ are Raney morphisms, 
					\begin{align*}
						(g \star f)(x \vee y) \le (g \circ f)(x \vee y) &= g(f(x) \vee f(y)) & &\text{since $x,y \in \B\S M_1$} \\ &= g f(x) \vee g f(y) & &\text{since $f(x),f(y) \in \B\S M_2$} \\ &= (g \star f)(x) \vee (g \star f)(y) & &\text{since $x,y \in \B\S M_1$}.
					\end{align*}
					\item For $a\in M_1$, we have 
					\begin{align*}
						(g\star f) (a) &=\bigvee \{g f(x) \mid x\in \B \S M_1, \, x\leq a\}\\
						&=\bigvee \{ (g \star f)(x)\mid x\in \B \S M_1, \, x\leq a\} &&\text{since $x\in \B \S M_1$}. && \qedhere
					\end{align*} 
				\end{enumerate}
			\end{proof}
			
			
			
			
			\begin{proposition}
				$\R : \RMT \to \Raney$ is a functor. 
			\end{proposition}
			
			\begin{proof}
				As we pointed out in the previous section, $\R M = (\S M,\O M)$ is a Raney extension for each MT-algebra $M$. Thus, $\R$ is 
				well defined on objects. 
				To see that it is well defined on morphisms, observe that if $f:M\to M'$ is a Raney morphism, then $f| _{\S M}:\S M\to\S M'$ is a coframe morphism by \ref{def: raneymor 1} and $f| _{\O M}:\O M\to\O M'$ is a frame morphism by \ref{def: raneymor 2}. Since 
				the restriction of $id_M$ is the identity on $\S M$, we have $\R(id_M) = 1_{\R M}$. 
				Let $f:M_1 \to M_2$ and $g:M_2 \to M_3$ be Raney morphisms. Since the restriction of $g \star f$ to $\S M_1$ is set-theoretic composition, 
				\[
				\R(g \star f)=
				(g \circ f)|_{\S M_1 }=g|_{\S M_2} \circ f|_{\S M_1}=\R(g) \circ \R(f).
				\]
				Thus, $\R : \RMT \to \Raney$ is a functor. 
			\end{proof}
			
			
			We conclude this section by connecting Raney morphisms with proximity morphisms  (see  \cref{def: proximitymor}).
			
			\begin{lemma}\label{welldefined from R to P}
				Let $f: M \to M'$ be a Raney morphism between MT-algebras. Define $\widehat f: M \to M'$ by 
				$
				\widehat f(a)= \bigvee \{f(x) \mid x \in \LC M, \, x \leq a\}.
				$
				Then $\widehat{f}$ is a proximity morphism.
			\end{lemma}
			\begin{proof}
				We show that $\widehat f$ satisfies \ref{def: proximitymor 1}--\ref{def: proximitymor 4}.
				
				\begin{enumerate}[label=\upshape(P\arabic*)]
					\item Since $\O M\subseteq \LC M$, we have 
					$\widehat{f}|_{\O M }= f|_{\O M }$ 
					and it is enough to apply  \ref{def: raneymor 2}. 
					\item Since $\widehat f(x)=f(x)$ for all $x\in \LC M$, for all $a \in M$, \[\widehat f(a)=\bigvee \{\widehat f(x)\mid x\in \LC M,x\leq a\}.\] Moreover, $\LC M$ is closed 
					under 
					binary meets and $\widehat f$ preserves binary meets from $\LC M$ (because $f$ does). Thus, \cref{l: preservation of all finite meets} applies, by which  $\widehat f$ preserves all binary meets.
					\item  Let $S \subseteq \LC M$ be finite.
					Since $\widehat f$ is order preserving by \ref{def: proximitymor 2}, it suffices to show that $\widehat f\left(\bigvee S\right) \leq \bigvee \left\{\widehat f(s) \mid s \in S \right\}$. We have
					\begin{align*}
						\widehat f\left(\bigvee S\right) &=\bigvee \left\{f(x) \mid x \in \LC M, \, x \leq \bigvee S\right\}\\
						& \leq \bigvee \left\{f(x) \mid x \in \B \S M, \, x \leq \bigvee S\right\}  && \LC M \subseteq \B \S M \\ &= f\left(\bigvee S\right) && \bigvee S \in \B\S M \\
						&=\bigvee \{f(s) \mid s \in S\} && \text {\ref{def: raneymor 4}}\\
						&= \bigvee \{\widehat f(s) \mid s \in S\} && \widehat f(x)=f(x) \text{ for each } x \in \LC M.
					\end{align*} 
					\item Using again that $\widehat f(x)=f(x)$ for each $x \in \LC M$, \ref{def: proximitymor 4} is immediate from the definition of $\widehat f$. \qedhere
				\end{enumerate}
			\end{proof}
			
			\begin{lemma}\label{composition from R to P}
				$ $
				\begin{enumerate}
					[ref=\thelemma(\arabic*)]
					\item  If $f:M_1 \to M_2$ and $g:M_2 \to M_3$ are Raney morphisms, then $\widehat{g \star f}=\widehat g \star \widehat f$.
					\item If $id_M : M \to M$ is an identity morphism in $\RMT$, then $\widehat{id_M} : M\to M$ is an identity morphism in $\PMT$.
				\end{enumerate}
				
			\end{lemma}
			\begin{proof}
				(1)  For $a \in M_1$, we have 
				\begin{align*}
					\left(\widehat{g} \star \widehat{f}\right)(a) &= \bigvee \{\widehat{g} \widehat{f}(x)  \mid  x \in \LC M_1, \, x \leq a \} \\
					&=\bigvee \{\widehat g f(x)  \mid x \in \LC M_1, \, x \leq a \} && \widehat f(x)=f(x) \text{ for each } x \in \LC M_1\\
					&=\bigvee \{g f(x)  \mid x \in \LC M_1, \, x \leq a \} && \widehat g f(x) =g f(x)  \text{ 
						by \cref{l:raneymorproperties3}} \\
					&=\bigvee \{(g \star f)(x) \mid x \in \LC M_1, \, x \leq a \} && (g \star f)(x)=g f(x)  \text{ since } x \in 
					\B \S M_1 \\
					&=\left(\widehat{g \star f}\right)(a).
				\end{align*} 
				
				(2) Since $id_M(x)=x$ for $x \in \LC M$, for each $a \in M$, we get 
				\[
				\widehat{id_M}(a)=\bigvee \{id_M(x) \mid x \in \LC M, \, x \leq a \}= \bigvee \{x \mid x \in \LC M, \, x \leq a \}= id^P_M(a). \qedhere
				\]

			\end{proof}
			
			Define $\mathcal{I}:\RMT \to \PMT$ 
			by setting $\mathcal{I} M=M$ for each MT-algebra $M$ and $\mathcal{I} f=\widehat f$ for each morphism $f$ in $\RMT$. It follows from \cref{welldefined from R to P,composition from R to P} that $\mathcal{I}$ is a functor. Let $\mathcal U:\Raney \to \Frm$ be the forgetful functor given by $\mathcal U R = L$ for each Raney extension $R=(C,L)$ and $\mathcal U f = f|_L$ for each $\Raney$-morphism $f$.
			
			\begin{theorem}
				The following diagram commutes.
				\[
				\begin{tikzcd}
					\RMT \ar[r, "\mathcal I"] \ar[d,"\R"'] & \PMT \ar[d, "\O"]\\
					\Raney \ar[r, "\mathcal U"] & \Frm
				\end{tikzcd}
				\]
			\end{theorem}
			
			
			\begin{proof}
				Let $M$ be an MT-algebra. 
				Then 
				\[
				\mathcal U \R M=\mathcal U (\S M , \O M) = \O M = \O \mathcal{I} M.
				\]
				Let $f:M \to M'$ be an $\RMT$-morphism. Then
				\[ 
				\mathcal U\R f = \mathcal U f|_{\S M} = f|_{\O M}=\widehat f|_{\O M}= \O \mathcal{I} f.
				\]
				Thus, the above diagram commutes.
			\end{proof}
			
			
			
			
			\section{Equivalence of $\RMT$ and $\Raney$} \label{sec: equivalence}
			
			
			As mentioned in \cref{Frames and MT-algebras}, the equivalence between $\MTP$ and $\TDMTP$ reflects the fact that isomorphisms in $\MTP$ are not order-isomorphisms, whereas this issue is resolved in $\TDMTP$. A similar situation arises between $\RMT$ and its full subcategory $\TMTR$ of $T_0$-algebras.
			In this section we show that $\R : \RMT \to \Raney$ is an equivalence of categories. This is done by proving that a quasi-inverse of $\R$ may be  constructed by lifting the Funayama envelope construction to a functor $\F : \Raney \to \RMT$. This equivalence 
			restricts to an equivalence between  $\Raney$ and $\TMTR$, where isomorphisms are order-isomorphisms. We thus think of the Funayama envelope as the $T_0$-hull of a Raney extension. 
			Finally, as promised in \cref{sec: Raney morphisms}, we provide an example of an $\Raney$-morphism that does not lift to an MT-morphism, thereby justifying the necessity of considering Raney morphisms. 

				Let $R = (C,L)$ be a Raney extension and let $M := \F R$ be its Funayama envelope. As we pointed out in \cref{sec: MT and RE}, $ L \cong \O M$. Therefore, since $L$ meet-generates $C$ and $\O M$ meet-generates $\S M$, 
					this extends to an isomorphism of Raney extensions $\rho_R : R \to \R \F R$.
					Consequently, we arrive at the following:
					
					\begin{theorem}\label{essentiallysurjective}
						The functor $\R:\RMT \to \Raney$ is essentially surjective.
					\end{theorem}
					
					\begin{remark} \label{rem: rho}
						We frequently identify $R$ with $\R \F R$ treating $\rho_R$ as the identity on $R$.
					\end{remark}
					
					
					
					We now show that the assignment $R \mapsto \F R$ is functorial. For this we recall that each bounded lattice morphism $h:A_1\to A_2$ between bounded distributive lattices lifts uniquely to a boolean morphism $\B h:\B A_1\to\B A_2$ between their boolean envelopes (see, e.g.,  \cite[Sec.~V.4]{balbes74}). For an MT-algebra $M$, the boolean envelope of $\S M$ is isomorphic to the boolean subalgebra $\B\S M$ of $M$ generated by $\S M$ (see, e.g., \cite[p.~99]{balbes74}). We will identify the boolean envelope of $\S M$ with this 
					subalgebra. 
					Thus, if $f:M_1 \to M_2$ is a Raney morphism, 
					\cref{l:raneymorproperties 2}  gives that $\B (f|_{\S M}) = f|_{\B \S M}$. 
					
					
					\begin{lemma}\label{frame to MT morphism}
						Let $R_1=(C_1,L_1)$ and $R_2=(C_2,L_2)$  be Raney extensions and $h:C_1 \to C_2$ an $\Raney$-morphism. Define $\F h:\F R_1 \to \F R_2$ by 
						\begin{eqnarray*}
							\F h(a) = \bigvee \{ \B h(x) \mid x \in \B C_1, \, x \leq a\}.
						\end{eqnarray*}
						Then $\F h$ is a Raney morphism. 
					\end{lemma}
					
					\begin{proof}
						We verify that $\F h$ satisfies 
						\cref{raneymor}.
						For $y \in \B C_1$, 
						\[
						\F h(y)=\bigvee \{ \B h(x) \mid x \in \B C_1, \, x \leq y \}=\B h(y).
						\]
						Thus $\F h| _{\B C_1} = \B h$. In particular, $\F h| _{C_1} =h$ and $\F h| _{L_1} =h|_{L_1}$. Therefore, \ref{def: raneymor 1} and \ref{def: raneymor 2} 
						hold. 
						By \cite[Lem.~4.8]{Bezh2010}, $\F h(a \wedge b)=\F h(a) \wedge \F h(b)$, and hence \ref{def: raneymor 3}
						holds. 
						For $x,y \in \B C_1$,
						\[\F h(x \vee y) = \B h(x \vee y)= \B h(x) \vee \B h(y) =\F h(x) \vee \F h(y),\] and thus \ref{def: raneymor 4} holds.
						Finally, 
						\[
						\F h(a) = \bigvee \{ \B h(x) \mid x \in \B C_1, \, x \leq a\} = \bigvee \{\F h(x) \mid x \in \B C_1, \, x \leq a\},
						\]
						and so \ref{def: raneymor 5}
						holds, yielding that $\F h$ is a Raney morphism.
					\end{proof}
					
					\begin{proposition}
						$\F: \Raney \to \RMT$ is a functor. 
					\end{proposition}
					\begin{proof}
						As we saw above, $\F$ is well defined both on objects and morphisms of $\Raney$. We show that $\F$ sends identity morphisms to identity morphisms and preserves composition. Let $R=(C,L)$ be a Raney extension and $a\in \F R$. Since $\B 1_R = 1_{\B C_1}$,  we obtain 
						\[\F 1_{R}(a) = \bigvee \{ \B 1_R(x) \mid x \in \B C_1, \, x \leq a\} = \bigvee \{ x \in \B C_1 \mid x \leq a\} = id_{\F R}(a).\] Therefore, $\F 1_{R} = id_{\F R}$. Next, let $f:C_1\to C_2$ and $g:C_2\to C_3$ be $\Raney$-morphisms between Raney extensions $R_1=(C_1,L_1)$, $R_2=(C_2,L_2)$, and $R_3=(C_3,L_3)$. Then 
						\begin{align*}
							(\F g \star \F f)(a) 
							&= \bigvee \{\F g \F f(x) \mid x \in \B C_1, \, x \leq a \}\\
							&= \bigvee \{\F g \B f(x) \mid x \in \B C_1, \, x \leq a \} && \F f| _{\B C_1} = \B  f\\
							&= \bigvee \{\B g \B f(x) \mid x \in \B C_1, \, x \leq a \} && \F g| _{\B C_2}= \B g; \, x \in \B C_1 \Longrightarrow \B  f(x) \in \B C_2\\
							&= \bigvee \{\B(g \circ f)(x) \mid x \in \B C_1, \, x \leq a \}&& \B g \circ \B  f = \B (g \circ f)\\
							&= \F(g \circ f)(a). &&
							\qedhere
						\end{align*}
					\end{proof}
					
					
					\begin{lemma}\label{adjoint equivalence}
						For an MT-algebra $M$, define $\zeta_M:\F \R M \to M$ by 
						\[
						\zeta_M(a)=\bigvee_M \{ x \in \B \S M \mid x \leq a \}
						\]
						and $\varphi_M : M \to \F \R M$ by 
						\[
						\varphi_M(b)=\bigvee_{\F \S M}\{x \in \B \S M \mid x \leq b\}.
						\] 
						Then $\zeta_M$ and $\varphi_M$ are mutually inverse Raney isomorphisms.
					\end{lemma}
					\begin{proof}
						Since $\zeta_M$ is identity on both $\B \S M$ and $\S M$, 
						it satisfies \ref{def: raneymor 5}, \ref{def: raneymor 4}, \ref{def: raneymor 1}, and \ref{def: raneymor 2}. It also satisfies \ref{def: raneymor 3} by \cref{l: preservation of all finite meets}. 
						Therefore, $\zeta_M$ is a Raney morphism. That $\varphi_M$ is a Raney morphism is proved similarly. It is left to show that $\zeta_M$ and $\varphi_M$ are mutually inverse in $\RMT$. Since ${\zeta_M(x) = \varphi_M(x) = x}$ for each $x\in \B \S M$, for $a \in M$, we have 
						\begin{eqnarray*}
							(\zeta_M \star \varphi_M) (a) &=& \bigvee_M \{\zeta_M \varphi_M(x) \mid x \in \B \S M, \, x \leq a \} \\
							&=& \bigvee_M \{x \in \B \S M \mid x \leq a\} \\
							&=& id_{M}(a);
						\end{eqnarray*}
						and for $b \in \F \S M$, we have 
						\begin{eqnarray*}
							(\varphi_M \star \zeta_M) (b) &=& \bigvee_{\F \S M} \{ \varphi_M \zeta_M(x) \mid x \in \B\S M, \, x \leq b \} \\
							&=& \bigvee_{\F \S M} \{x \in \B \S M \mid x \leq b\} \\
							&=& id_{\F \R M}(b), 
						\end{eqnarray*}
						concluding the proof.
					\end{proof}
					
					\begin{lemma}\label{natural transformation}
						$ $
						\begin{enumerate}
							\item $\rho: 1_{\Raney} \to \R \F$ is a natural transformation. 
							\item $\zeta : \F  \R \to 1_{\RMT}$ is a natural transformation.
						\end{enumerate}
						
					\end{lemma}
					
					\begin{proof}
						(1) Let $f:C_1\to C_2$ be an $\Raney$-morphism 
						between Raney extensions $R_1=(C_1,L_1)$ and $R_2=(C_2,L_2)$. We must show that the following diagram commutes.
						\[\begin{tikzcd}[column sep=5em]
							C_1 \ar[r, "f"] \ar[d, "\rho_{R_1}" ']& C_2 \ar[d, "\rho_{R_2}"]\\
							\R\F R_1 \ar[r, "\R\F f"']  &  \R \F R_2 
						\end{tikzcd}\]
						For $i=1,2$, we identify $C_i$ with $\rho_{R_i}[C_i]$ and assume that $C_i\subseteq \F R_i$ (see \cref{rem: rho}). Since
						the functor $\R$ sends a Raney morphism 
						to its restriction to the coframe of saturated elements,
						commutativity of the diagram amounts to showing that $\F f(a)=f(a)$ for each $a\in C_1$, which follows from the definition of $\F f$. 
						
						(2) Let $g:M_1 \to M_2$ be a Raney morphism between MT-algebras. We must show that the following diagram commutes. 
						\[\begin{tikzcd}[column sep=5em]
							M_1 \ar[r, "g"] & M_2 \\
							\F \R M_1 \ar[r, "\F\R g"'] \ar[u, "\zeta_{M_1}"] &  \F \R M_2 \ar[u, "\zeta_{M_2}"']
						\end{tikzcd}\]
						First, let $x \in \B \S M_1$. Then $g(x) \in \B \S M_2$ by \cref{l:raneymorproperties 2}. Thus, 
						\begin{align*}
							\zeta_{M_2} \F\R g(x)  
							&=\zeta_{M_2}\F \left(g|_{\S M_1}\right)(x) \\
							&= \zeta_{M_2} \B \left(g|_{\S M_1}\right)(x)\\
							&= \zeta_{M_2}g|_{\B \S M_1}(x) && 
								\B \left(g|_{\S M_1}\right) =  g|_{\B\S M_1} \\
								&= \zeta_{M_2}g(x)\\
								&= g(x) &&    \zeta_{M_2}g(x)=g(x) \\
								& = g\zeta_{M_1}(x) && \zeta_{M_1}(x)=x.
							\end{align*}
							Next, let  $a \in \F\S M_1$. Then, by the above, 
							\begin{align*}
								(\zeta_{M_2} \star \F\R g)(a) &= \bigvee \{\zeta_{M_2}\F\R g(x) \mid x \in \B \S M_1, \, x \leq a \}\\
								&= \bigvee \{g\zeta_{M_1}(x) \mid x \in \B \S M_1, \, x \leq a \} = (g \star \zeta_{M_1})(a). \qedhere
							\end{align*}

						\end{proof}
						
						\begin{theorem}\label{adjunction of raney and MT}
							The functors $\R$ and $\F$ 
							establish an equivalence of $\RMT$ and $\Raney$.
						\end{theorem}
						\begin{proof} 
							By \cref{natural transformation}, $\rho$ and $\zeta$ are natural transformations. By \cref{adjoint equivalence}, $\zeta$ is an isomorphism on all components, and the same is true for $\rho$ by the paragraph before  \cref{essentiallysurjective}. 
							Thus, it suffices to show that these are
							the unit and counit of the adjunction $\F\dashv \R$. 
							
							Let $M$ be an MT-algebra. 
							In view of our identifications, $\R \zeta_{M}$ and $\rho_{\R M}$ are identities. 
							Hence, for $s \in \S M$, we have 
							\[
							(\R \zeta_{M} \circ \rho_{\R M})(s)=\R \zeta_M(s)=s.
							\]
							
							Let $R=(C,L)$ be a Raney extension. 
							Again, by our identifications,
							$\rho_{R}$  
							and $\B \rho_{R}$ are identities. 
							Therefore, for $x \in \B C$,
							\[(\zeta_{\F R} \circ \F\rho_{R})(x)=\zeta_{\F R} \B\rho_{R}(x)=\zeta_{\F R}(x)=x.\]
							Thus, for $a \in \F R$,
							\begin{eqnarray*}
								(\zeta_{\F R} \circ \F\rho_{R})(a) &=& \bigvee \{
								\zeta_{\F R}\F\rho_{R}(x) \mid  x \in \B C, \, x \leq a \} \\ &=& \bigvee \{x \in \B C \mid x \leq a \} = a,
							\end{eqnarray*}
							concluding the proof.
						\end{proof}
						
						
						As with proximity morphisms between MT-algebras \cite[Ex.~3.14]{bezhanishvili2025funayamaenvelopetdhullframe}, isomorphisms in $\RMT$ are not structure-preserving bijections. In fact, the same example 
						works because in the finite case, open and saturated elements coincide. Since $\R:\RMT\to \Raney$ is an equivalence of categories, from \cite[Prop.~7.47]{AHS2006} we obtain the following characterization of isomorphisms in $\RMT$: 
						
						\begin{proposition}\label{l:isochar}
							Let $f:M\to M'$ be a Raney morphism between MT-algebras.
							\begin{enumerate} [ref=\theproposition(\arabic*)]
								\item $f$ is an isomorphism iff $\R f$ is an isomorphism of Raney extensions. \label [proposition]{iso}
								\item $f$ is a monomorphism iff $\R f$ is a monomorphism of Raney extensions. \label [proposition]{mono}
								\item $f$ is an epimorphism iff $\R f$ is an epimorphism of Raney extensions. \label [proposition]{epi}
							\end{enumerate} 
						\end{proposition}
						
						In particular, in $\RMT$ there exist monomorphisms that are not injective and epimorphisms that are not surjective (again, see \cite[Ex.,3.14]{bezhanishvili2025funayamaenvelopetdhullframe}). This counterintuitive behavior disappears once we restrict our attention to $T_0$-algebras.
						\begin{proposition}\label{prop: proximity isos}
							A Raney morphism $f:M\to M'$ between $T_0$-algebras is an $\RMT$-isomorphism 
							iff it is an order-isomorphism.
						\end{proposition}
						
						\begin{proof}
							First suppose 
							$f:M\to M'$ is an $\RMT$-isomorphism between $T_0$-algebras. By  \cref{iso}, $\R f$ is an $\Raney$-isomorphism. 
								Since $\Raney$-isomorphisms are 
								coframe isomorphisms, 
								$\B \R f:\B \S M\to \B \S M'$ is a boolean isomorphism. Therefore, it can be lifted to an order-isomorphism between $\F\S M$ and $\F\S M'$ (see, e.g., \cite[Thm.~7.41(ii)]{Davey2002}), which coincides with $\F \R f$ since it preserves arbitrary joins.
								Because $M$ and $M'$ are $T_0$-algebras, they are order-isomorphic to $ \F \R M$ and $\F \R M'$, respectively (see \cref{lem: Raney is t0}). Thus, up to order-isomorphism, $f= \F \R f$, yielding that $f$ is an order-isomorphism.
								
								Conversely, suppose 
								$f:M\to M'$ is an order-isomorphism. Then its inverse $f^{-1}:M'\to M$ is an order-isomorphism. Therefore, for $a\in M$, we have
								\begin{align*}
									(f^{-1}\star f)(a)=\bigvee \{f^{-1} f(x)\mid x\in \B \S M, \, x\leq a\}=\bigvee \{x\in \B \S M  \mid x\leq a\}=id_M(a).
								\end{align*}
								A similar argument yields that $f\star f^{-1}=id_{M'}$. Thus, $f$ is an $\RMT$-isomorphism. 
							\end{proof}
							
							
							
							
							
							
							\color{black}

							Let $\TMTR$ be the full subcategory of $\RMT$ consisting of $T_0$-algebras. We record the following structural features of $\TMTR$:
							
							\begin{proposition}
								\label{rem: T0MTR}
								$ $
								\begin{enumerate}[ref=\theproposition(\arabic*)]
									\item Identities in $\TMTR$ are identity functions.
									\item Each MT-morphism between $T_0$-algebras is a Raney morphism. 
									\label[proposition]{rem: T0MTR 2}
								\end{enumerate}
							\end{proposition}
							
							\begin{proof}
								(1) Let $M$ be an MT-algebra. By \cref{prop: char of T0}, 
								\[
								M \mbox{ is a $T_0$-algebra} \iff \forall a \in M, \ a=\bigvee \{ x \in \B \S M \mid x \le a \} \iff \forall a \in M, \ a = id_M(a). 
								\]
								
								(2) Let $f:M\to M'$ be an MT-morphism between $T_0$-algebras.  
								Then $f$ satisfies \ref{def: raneymor 1}--\ref{def: raneymor 4}. To see that it satisfies \ref{def: raneymor 5},
								let $a\in M$. Since $M$ is a $T_0$-algebra, ${a=\bigvee \{x\in \B \S M \mid x\leq a\}}$. Because $f$ preserves all joins, 
								\[
								f(a)=\bigvee \{f(x)\mid x\in \B \S M, \, x\leq a\}.
								\]
								Thus, $f$ is a Raney morphism. \qedhere
								
								
							\end{proof}
							
							We also 
							obtain the following  analogue of \cref{MTpequivalenttoframes 2} for $T_0$-algebras and Raney extensions.
							
							\begin{theorem}\label{WMT_D&Frm}
								The equivalence 
								$\R : \RMT \leftrightarrows \Raney : \F$ restricts to an equivalence between $\TMTR$ and $\Raney$. Consequently, the reflector $\F\R : \RMT \to \TMTR$ is an equivalence.
							\end{theorem}
							
							\begin{proof}
								By \cref{lem: Raney is t0}, 
								the equivalence $\R : \RMT \leftrightarrows \Raney : \F$ of \cref{adjunction of raney and MT} restricts to an equivalence between $\TMTR$ and $\Raney$. 
								Thus, the reflector $\F\R : \RMT \to \TMTR$ is an equivalence.
							\end{proof}

								We conclude by giving an example of an $\Raney$-morphism between Raney extensions that does not lift to an MT-morphism between their Funayama envelopes, thus justifying the need for the notion of Raney morphism between MT-algebras. 
								
								Note that if $L$ is both a frame and a coframe, then 
								the pair $(L,L)$ is a Raney extension. Moreover, if $L_1,L_2$ are such
								and $f : L_1 \to L_2$ is a complete lattice morphism, then $f$ is an $\Raney$-morphism between the Raney extensions $(L_1,L_1)$ and $(L_2,L_2)$. Thus, it is enough to show that not every such $f$ lifts to a complete boolean morphism $\F f : \F L_1 \to \F L_2$. 
								
								
								\begin{example} \label{Cantor example}
									Let $L_1$ be the Cantor set. Since $L_1$ is a closed subset of $[0,1]$, it is closed under arbitrary suprema and infima. Therefore, $L_1$ is a complete lattice in the order inherited from $[0,1]$. 
									Thus, since $L_1$ is a chain, it is both a frame and a coframe. 
									
									There are various representations of $L_1$. For our purposes, we think of $L_1$ as  
									\[
									L_1 = \{ 0.a_1 a_2 a_3 \ldots \mid a_i \in \{0, 2\} \}
									\]
									(see, e.g., \cite[p.~320]{GH09}).
									%
									In this representation, the right endpoints of a removed interval in the construction of the Cantor set 
									are the finite sequences in $\{0,2\}$ followed by an infinite tail of $0$s:
									\[
									\mathcal{R}(L_1) = \left\{ 0.a_1 a_2 \ldots a_n \overline{0} \;\middle|\; a_i \in \{0, 2\},\; n \in \mathbb{N},\; (a_1,\dots, a_n) \neq (0,\dots, 0)\right\},
									\]
									and the left endpoints are the same finite sequences followed by an infinite tail of $2$s:
									\[
									\mathcal{L}(L_1) = \left\{ 0.a_1 a_2 \ldots a_n \overline{2} \;\middle|\; a_i \in \{0, 2\},\; n \in \mathbb{N}.\; (a_1,\dots, a_n) \neq (2,\dots, 2) \right\}
									\]
									(see, e.g., \cite[p.~535]{GH09}).
									
									It is straightforward to check that $\mathcal{L}(L_1)$ meet-generates $L_1$ (indeed, each $x = 0.a_1 a_2 a_3 \ldots$ is the meet of the $x_n := 0.a_1 a_2 \ldots a_n \overline{2} \in \mathcal{L}(L_1)$). Also, since each left endpoint is covered by its corresponding right endpoint, it is clear that no element in $\mathcal{L}(L_1)$ is a meet of elements outside of $\mathcal{L}(L_1)$.
									
									We let $L_2 = L_1 \setminus \mathcal{L}(L_1)$. By the above observation, $L_2$ is closed under arbitrary meets. Since $L_1$ is a chain, for each $a,b \in L_1$, the relative pseudocomplement $a \to b := \bigvee \{ x \in L_1 \mid a \wedge x \le b \}$ is calculated by the following simple formula:
									\[
									a \to b =
									\begin{cases}
										1 & a \le b, \\
										b & a > b.
									\end{cases}
									\]
									Therefore, $a \to b \in L_2$ for each $a \in L_1$ and $b \in L_2$. Thus, $L_2$ is a sublocale of $L_1$ (see, e.g., \cite[p.~26]{PicadoPultr2012}). Consequently, $L_2$ is a complete lattice, and since $L_2$ is a chain, it is both a frame and a coframe. 
									
									Each sublocale $S$ of a frame $L$ induces the nucleus $j : L \to L$, given by $ja = \bigwedge (\up a \cap S)$, whose fixpoints are $S$ (see, e.g., \cite[p.~32]{PicadoPultr2012}). Observe that the nucleus $j$ on $L_1$ corresponding to the sublocale $L_2$ of $L_1$ is given by 
									\[
									ja =
									\begin{cases}
										a & a \in L_2, \\
										b & a \in \mathcal{L}(L_1),
									\end{cases}
									\]
									where $b$ is the unique cover of $a \in \mathcal{L}(L_1)$. We now show that the corresponding frame surjection $j : L_1 \to L_2$ is a complete lattice morphism. For this
									it is sufficient to show that  $j$ preserves arbitrary meets. Let $S \subseteq L_1$ and $a = \bigwedge S$. First suppose that $a \in \mathcal{L}(L_1)$. Then $b$ is its unique cover, and hence $a$ must be the least element of $S$. But then $j (\bigwedge S) = ja = \bigwedge \{ js \mid s \in S\}$. Next suppose that $a \notin \mathcal L(L_1)$. If $a \in S$, then we are done. Otherwise, for each $s \in S$ there is $t \in S$ with $a < t < s$, and hence $jt \le s$ since $jt = t$ or $jt$ covers $t$. Thus, $j(\bigwedge S) = ja = a = \bigwedge \{ js \mid s \in S \}$, and hence $j$ is a complete lattice morphism. Consequently, $j$ is an $\Raney$-morphism between the Raney extensions $(L_1,L_1)$ and $(L_2,L_2)$. It is left to show that $j$ does not lift to a complete boolean morphism between their Funayama envelopes. For this, we utilize a result from \cite{manuell15} (see also \cite{arrieta22}) that characterizes when such a lift is possible. 
									
									For any sublocale $S$ of a frame $L$ with the corresponding frame surjection $j : L \to S$, by \cite[Cor.~4.2]{arrieta22} (see also \cite[Lem.~3.39]{manuell15}) $j$ lifts to a complete boolean morphism of the Funayama envelopes iff $S$ is a join of locally closed sublocales. Here we recall (see, e.g., \cite[p.~33]{PicadoPultr2012})  that each $a \in L$ gives rise to the \emph{open sublocale} $\O(a)=\{a \to x \mid x \in L\}$, the \emph{closed sublocale} $\C(a)=\upset a$, and a sublocale of $L$ is {\em locally closed} provided it is of the form $\O(a) \cap \C(b)$ for some $a,b \in L$. 
									To see that $L_2$ is not a join of locally closed sublocales, we will show that the only locally closed sublocale contained in $L_2$ is 
									$\{1\}$.
									For $a,b \in L_1$, we have
									\[
									\mathsf{O}(a) = [0,a) \cup \{1\} \quad \text{and} \quad
									\mathsf{C}(b) = [b,1].
									\]
									Therefore, 
									\[
									\mathsf{O}(a) \cap \mathsf{C}(b) =
									\begin{cases}
										\{1\}, & a \leq b, \\
										[b, a) \cup \{1\}, & b < a.
									\end{cases}
									\]
									In the second case, since $\mathcal{L}(L_1)$ meet-generates $L_1$, there is $x \in \mathcal{L}(L_1)$ with $b \leq x < a$. 
									Hence,  $x \in \mathsf{O}(a) \cap \mathsf{C}(b)$ but $x \notin L_2$, so the intersection is not contained in $L_2$. 
									Consequently, $\F j$ is not a complete boolean morphism.
								\end{example}

								\section*{Declaration of funding}
								Anna Laura Suarez received financial support from the Centre for Mathematics of the University of Coimbra (funded by the Portuguese Government through FCT project UIDB/00324/2020).
								
								\section*{Acknowledgements}
								We would like to thank the referee for careful reading and positive feedback.
								

								\bibliographystyle{alpha}
								\bibliography{referenceraney}
								
							\end{document}